%% file: PAN_21_08_2025.tex
\documentclass[a4paper, reqno, 12pt]{amsart}

\newif\iffinal

\long\def\comment#1{\relax}

\pdfoutput=1

\usepackage[utf8x]{inputenc}
\usepackage[T1]{fontenc}
\usepackage{lmodern}
\usepackage[english]{babel}

\usepackage{amssymb}
\usepackage{colonequals}
\usepackage{mathtools}
\usepackage{hyperref}
\usepackage{colortbl}
\usepackage{color}
\usepackage{stmaryrd}
\usepackage{multirow}

\newcommand{\N}{\mathbb{N}}
\newcommand{\Z}{\mathbb{Z}}
\newcommand{\Q}{\mathbb{Q}}
\newcommand{\cO}{\mathcal O}
\newcommand{\cB}{\mathcal B}
\newcommand{\cF}{\mathcal F}
\newcommand{\fp}{\mathfrak p}
\newcommand{\fq}{\mathfrak q}
\newcommand{\fr}{\mathfrak r}
\DeclareMathOperator{\q}{\mathbf{q}}  
\newcommand{\val}{\mathsf v}

\newcommand{\Mod}[1]{\ (\mathrm{mod}\ #1)}
\def\card#1{\left|#1\right|}
\def\abs#1{\left|#1\right|}

\newtheorem{proposition}{Proposition}[section]
\newtheorem{theorem}[proposition]{Theorem}
\newtheorem{corollary}[proposition]{Corollary}
\newtheorem{fact}[proposition]{Fact}
\newtheorem{lemma}[proposition]{Lemma}
\theoremstyle{definition}
\newtheorem{definition}[proposition]{Definition}

\newtheorem{example}[proposition]{Example}
\newtheorem{remark}[proposition]{Remark}

\DeclareMathOperator{\Int}{Int}
\DeclareMathOperator{\divides}{\big\vert}
\DeclareMathOperator{\dividess}{\vert}
\DeclareMathOperator{\dividesnot}{\not\big\vert}
\DeclareMathOperator{\dividenot}{\not\vert}

\DeclareMathOperator{\supp}{supp}
\DeclareMathOperator{\fd}{fd}  

\makeatletter
\@namedef{subjclassname@2020}{\textup{2020} Mathematics Subject Classification}
\makeatother

\usepackage{letltxmacro}
\usepackage{todonotes}
\LetLtxMacro\todonotestodo\todo
\renewcommand{\todo}[2][]{\todonotestodo[backgroundcolor=yellow, #1]{TODO: {#2}}}
\newcommand{\attention}[2][]{\todonotestodo[backgroundcolor=red, #1]{TODO: {#2}}}
\iffinal\renewcommand{\todo}[2][]{}\renewcommand{\attention}[2][]{}\fi

\author[V. Fadinger]{Victor Fadinger}
\address{Institute for Mathematics and Scientific Computing\\
  University of Graz\\ Heinrichstraße 36, \\8010 Graz\\Austria}
\email{\href{mailto: viktor.fadinger@gmail.com}{viktor.fadinger@gmail.com}}
\thanks{V.~Fadinger was supported by the Austrian Science Fund (FWF) [10.55776/W1230].}

\author[S.~Frisch]{Sophie Frisch}
\address{Institut für Analysis und Zahlentheorie\\
Graz University of Technology\\
Kopernikusgasse 24\\8010 Graz\\Austria}
\email{\href{mailto:frisch@math.tugraz.at}{frisch@math.tugraz.at}}
\thanks{S.~Frisch (the corresponding author) and D.~Windisch acknowledge
support by the Austrian Science Fund (FWF) [10.55776/P30934].}

\author[S.~Nakato]{Sarah Nakato}
\address{Department of Mathematics \\
 Kabale University\\
 Plot 364 Block 3 Kikungiri Hill
Kabale\\ Uganda}
\email{\href{mailto: snakato@kab.ac.ug}{snakato@kab.ac.ug}}
\thanks{S.~Nakato was supported in part by Africa-UniNet under the
project P142\_Uganda: Uganda-Austria Collaboration in Algebra and Geometry}

\author[D.~Smertnig]{Daniel Smertnig}
\address{Faculty of Mathematics and Physics (FMF)\\
University of Ljubljana
and Institute of Mathematics, Physics and Mechanics (IMFM)\\
Jadranska ulica 21\\
1000 Ljubljana, Slovenia}
\email{\href{mailto:daniel.smertnig@fmf.uni-lj.si}%
{daniel.smertnig@fmf.uni-lj.si}}
\thanks{D.~Smertnig was funded in whole or in part by the
Slovenian Research and Innovation Agency (ARIS):
Grants P1-0288, J1-60025 and the Austrian Science Fund (FWF)
[10.55776/P36742].}

\author[D.~Windisch]{Daniel Windisch}
\address{KU Leuven, Celestijnenlaan 200B\\
B-3001 Leuven (Heverlee) \\Belgium}
\email{\href{mailto:daniel.windisch.math@gmail.com}%
{daniel.windisch.math@gmail.com}}

\title[{Primes and irreducible elements in atomic domains}]
{Primes and absolutely or non-absolutely irreducible elements in atomic domains}


\begin{document}
\vspace*{-3ex}
\maketitle
\begin{abstract}
We give examples of atomic integral domains satisfying each of the eight
logically possible combinations of existence or non-existence of
the following kinds of elements: 1) primes, 2) absolutely irreducible
elements that are not prime, and 3) irreducible elements that are not
absolutely irreducible. A non-zero non-unit is called absolutely
irreducible (or, a strong atom) if every one of its powers factors
uniquely into irreducibles.
\end{abstract}

\input{PAN_mmp_mmm_intro_19_08_2025}

\input{PAN_prelim_20_08_2025}

\input{PAN_pmm_NA_19_08_2025}

\input{PAN_ppm_ANANP_19_08_2025}

\input{PAN_pmp_NAPNoA_19_08_2025}
\input{PAN_ppp_PANA_21_08_2025}

\input{PAN_mpm_NPA_19_08_2025}

\input{PAN_mpp_PA_19_08_2025}

\textbf{Acknowledgment.} The authors wish to thank the second of two
anonymous referees for contributing an additional example of a ring
containing primes and both absolutely and non-absolutely irreducible
elements to Section~\ref{section:PANA}, namely, the one in
Example~\ref{pana}.
\bibliographystyle{plain}
\bibliography{bibliography_PAN}

\end{document}

%% file: PAN_mmp_mmm_intro_19_08_2025.tex
\section{Introduction}\label{section:intro}
Among atomic domains, that is, domains in which every non-zero
non-unit is a product of irreducibles, unique factorization domains
are characterized by the fact that all irreducibles are prime.

Chapman and Krause \cite{ChKr2012:Atomic-decay} showed for rings of
integers in number fields that ${\mathcal{O}}_K$ is a UFD if and only if
every irreducible element is absolutely irreducible --- meaning that each
of its powers factors uniquely into irreducibles --- a weaker property
than prime.

Their result prompts the question whether this characterization of
unique factorization domains holds in greater generality: among Dedekind
domains, for instance.  The answer is no.

We have the following implications (whose converses do not hold):
\[
\text{prime}\quad\Longrightarrow\quad\text{absolutely irreducible}
\quad\Longrightarrow\quad\text{irreducible}
\]

This gives us three kinds of elements that may or may not exist in
a given domain:
\begin{enumerate}
\item
Primes
\item
Absolutely irreducibles that are not prime
\item
Irreducibles that are not absolutely irreducible
\end{enumerate}
--- and, therefore, eight logically possible combinations of existence or
non-existence of each kind of element.

Monoid examples for all eight scenarios are easy to find, compared
to examples in integral domains. For instance,
Baginski and Kravitz~\cite{BaKr2010:HFKR} provide examples of
non-factorial monoids whose irreducibles are all absolutely irreducible,
but not prime.

We will show that all eight scenarios occur in atomic domains.

\begin{remark}\label{mmp-mmm}
Some cases are trivial: Atomic domains without any irreducible
elements, let alone absolutely irreducible or prime elements, are
just fields. (There are also non-atomic domains without irreducible
elements, the so-called antimatter domains~\cite{anderson2007antimatter}.)

Atomic domains containing primes but no other irreducible elements are
precisely those UFDs that are not fields, as mentioned.
\end{remark}

For the six non-trivial combinations we now proceed
to give examples.

In the following table, plus indicates existence, and minus,
non-existence. $R_1$ and $R_2$ are certain Dedekind domains
with class group $\Z^n$, see Proposition~\ref{Prop:withoutprime}
and Proposition~\ref{prop:exm-allai-but-not-ufd}, respectively.

\begin{table}[h]
\centering

\resizebox{1\textwidth}{!}{
\begin{tabular}{|c|c|c|c|c|}
\hline
\multirow{2}*{Result}& \multirow{2}*{Example}  & Irreducible but not &
Absolutely irreducible  &  \multirow{2}*{Prime} \\
 &  &absolutely irreducible &but not prime& \\ \hline

\ref{pana1},\ref{pana2}& $\Z[\sqrt{-14}]$ & + & +& + \\ \hline

\ref{ppm}& $\Int(\cO_K)$ & + & +& - \\ \hline

\ref{pmp},\ref{pmpOne}& $\mathbb{R} + X\mathbb{C}[X]$ & + & -& +\\ \hline

\ref{pmm}& $\mathbb{R} + X\mathbb{C}[[X]]$ & + & -& - \\ \hline

\ref{mpp}& $R_2$  & - & + & + \\ \hline

\ref{mpm}& $R_1$ & - & + & - \\ \hline

\ref{mmp-mmm}& UFDs& - & - & + \\ \hline

\ref{mmp-mmm}& Fields & -& -& - \\ \hline

\end{tabular}}
\end{table}

%% file: PAN_prelim_20_08_2025.tex
\section{Preliminaries}

\subsection{Factorization terms}
We recall some concepts and terminology related to factorization.
For a comprehensive introduction to non-unique factorizations, we refer
to the textbook by Geroldinger and Halter-Koch~\cite{GeHa2006:NonUniq}.

The terms that we here define for a monoid $H$ we will use mostly
(but not only) in the context of an integral domain $R$. In that case,
the monoid in question is understood to be $(R\setminus\{0\},\cdot)$.

\begin{definition}
Let $(H,\cdot)$ be a commutative monoid.
\begin{enumerate}
\item
We denote the group of units of the monoid $H$ by $H^{\times}$.
\item $r\in H$ is said to be \emph{irreducible} in $H$
(or, an \emph{atom} of $H$) if it is a non-unit that is
not a product of two non-units of $H$.
\item A \emph{factorization} of a non-unit $r \in H$ is an
expression
       \[r = a_{1}\cdot\ldots\cdot a_{n},\]
       where $n\ge 1$ and $a_i$ is irreducible in $H$ for $1\le i \le n$.

The number $n$ of irreducible factors is called the
\emph{length} of the factorization. 
\item $r,s\in H$ are \emph{associated} in $H$ if there exists
a unit $u \in H$ such that $r = us$. We denote this by $r \sim s$.
\item Two factorizations into irreducibles of the same element,
       \begin{equation}\label{eq:2-fac-same-diff}
       r = a_{1}\cdot\ldots\cdot a_{n} = b_{1} \cdot\ldots\cdot b_{m},
       \end{equation}
are called \emph{essentially the same} if $n = m$ and, after
re-indexing,
$a_{j}\sim b_{j}$ for $1 \leq j \leq m$.
Otherwise, the factorizations in
\eqref{eq:2-fac-same-diff} are called \emph{essentially different}.
\item
$(H,\cdot)$ is called \emph{atomic} if every non-unit has a
factorization.
\item $(H,\cdot)$ is \emph{factorial} if $H$ is atomic and any 
two factorizations of an element are essentially the same.
\item $(H,\cdot)$ is \emph{half-factorial} if $H$ is atomic and any 
two factorizations of an element have the same length, i.e., the same
number of irreducible factors.
\item In a half-factorial monoid, the \emph{length} of a
non-zero element $h$, denoted $\ell(h)$, is defined as the length of
a factorization of $h$ into irreducibles. 
\end{enumerate}
\end{definition}

\begin{definition}\label{defcancellative}
A commutative  monoid $(H,\cdot)$ is called cancellative
if $ab=ac$ implies $b=c$, for any $a,b,c\in H$. 

The quotient group $\q(H)$ of a cancellative monoid is the group
defined on the set of equivalence classes of pairs $(a,b)\in H\times H$ 
with respect to the equivalence relation 
\[(a,b)\simeq (c,d) \Longleftrightarrow ad=bc,\]
endowed with the multiplication \[\frac{a}{b}\cdot\frac{c}{d} =
\frac{ab}{cd}, \]
where $\frac{a}{b}$ denotes the equivalence class of $(a,b)$.
\end{definition}

\begin{definition}\label{defabsirred}
Let $H$ be a cancellative commutative monoid.
\begin{enumerate}
\item
An irreducible element $r \in~H$ is called \emph{absolutely irreducible}
(or, a \emph{strong atom}), if for all natural numbers $n$, every
factorization of $r^n$ is essentially the same as $r^n = r \cdot\ldots\cdot r$.
\item
If $r \in H$ is irreducible, but not absolutely irreducible,
it is called \emph{non-absolutely irreducible}.
\end{enumerate}
\end{definition}

\subsection{Transfer homomorphisms}
Transfer homomorphisms are a key tool in factorization theory. 
They are used to study non-unique factorization in a domain 
(or monoid) using a simpler ``model'' monoid.
In this section, we show that transfer homomorphisms preserve 
absolute irreducibility (in the forward direction).

\begin{definition}\cite[Definition 3.2.1]{GeHa2006:NonUniq}
Let $H$ and $M$ be cancellative commutative monoids. A monoid homomorphism
$\theta\colon H \longrightarrow M$ is called a transfer homomorphism
if it has the following properties:
\begin{enumerate}
\item $M = \theta(H)M^{\times}$ and $\theta^{-1}(M^{\times})=  H^{\times}$
\item If $h \in H$ and $b$,~$c \in M$ such that $\theta(h) = bc$,
 then there exist $v$,~$w \in H$
such that $h=vw$ and $\theta(v) \sim b$ and $\theta(w) \sim c$.
\end{enumerate}
\end{definition}

\begin{fact}\cite[Proposition 3.2.3]{GeHa2006:NonUniq}\label{fact:TH}
Let $\theta\colon H \longrightarrow M$ be a transfer homomorphism.

\begin{enumerate}
\item \label{TH:irred} An element $u \in H$ is irreducible in $H$
if and only if $\theta(u)$ is irreducible in $M$.
\item $H$ is atomic if and only if $M$ is atomic.
\end{enumerate}
\end{fact}

\begin{lemma}\label{th-absirred:forward}
Let $\theta\colon H \longrightarrow M$ be a transfer homomorphism,
and $c \in H$ an irreducible element.
If $c$ is absolutely irreducible in $H$
then $\theta(c)$ is absolutely irreducible in $M$.
\end{lemma}

\begin{proof}
If $\theta(c)$ is not absolutely irreducible in $M$,
then there exists an irreducible element $a\in M$, not associated 
to $\theta(c)$, that divides $\theta(c)^m$ for some $m \in \N$.
By the first of the defining properties of a transfer homomorphism, 
we may assume $a=\theta(b)$ for some $b\in H$. Since $\theta(b)$ is
irreducible, it follows by the second of the defining properties of
a transfer homomorphism that this $b$ is an irreducible element of $H$,
dividing $c^m$. Also, $b$ cannot be associated to $c$, because otherwise
$\theta(b)$ would be associated to $\theta(c)$.
\end{proof}

\begin{remark}\label{l:th-absirred}
The converse of Lemma~\ref{th-absirred:forward} does not hold.
A non-absolutely irreducible element of $H$ may be
mapped to an absolutely irreducible element of $M$ by a transfer
homomorphism $\theta\colon H \longrightarrow M$. We illustrate this
by an example:
If $H$ is a half-factorial commutative monoid and $M = (\N_0, +)$, 
then the function
\begin{align*}
 \theta\colon & H \longrightarrow M   \\
  & a \longmapsto \ell(a),
\end{align*}
where $\ell(a)$ denotes the length of $a$ (that is, the number of
irreducibles in a factorization of $a$), is a transfer homomorphism.

If $c \in H$ is irreducible in $H$, then $\theta(c)$ is irreducible in $M$,
by Fact~\ref{fact:TH}.
The unique irreducible of $M$ is $1$ (which is prime).
Hence $\theta(c)=1$ for all irreducible $c \in H$, and in particular,
every irreducible in $M$ is absolutely irreducible.
However, there exist half-factorial monoids $H$ that contain
non-absolutely irreducibles, for instance 
$H = \mathcal O_K\setminus \{0\}$ with $\mathcal O_K$ a ring of 
algebraic integers whose class group is isomorphic to $\Z/2\Z$ 
(see Fact~\ref{f:algint} below).
\end{remark}

\subsection{Krull monoids}
Recall that an integral domain is Krull if and only if it is
completely integrally closed and Mori. The concepts involved in
this characterization depend only on the multiplicative monoid
and can thus be used to define Krull monoids.

\begin{def}\label{defKrullmonoid}
A cancellative commutative monoid $H$ is called 
\emph{completely integrally closed} if it contains all elements 
of $\q(H)$ that are almost integral over $H$, that is, those 
$c\in \q(H)$ for which there exists some $d\in H$ such that for
all $n\in\N$, $dc^n\in H$.

Likewise, a  cancellative commutative monoid $(H,\cdot)$ is \emph{Mori}
if it satisfies the ascending chain condition for  divisorial ideals. 
Here, a divisorial ideal is defined just like a divisorial ideal of 
an integral domain, with the quotient group $\q(H)$ of the cancellative 
monoid taking the place of the quotient field of an integral domain.

A cancellative commutative monoid is a \emph{Krull monoid} if it 
is completely integrally closed and Mori.
\end{def}

With this definition of Krull monoid, an integral domain $D$ is 
Krull if and only if $D \setminus \{0\}$ is a Krull
monoid~\cite[Thm.~2.10.2.3]{GeHa2006:NonUniq}.

We refer to \cite[Chapter 2]{GeHa2006:NonUniq} for the algebraic 
theory of Krull monoids and for more details on the terminology we
just introduced.

Let $H$ be a Krull monoid and let $\mathfrak X(H)$ denote the set 
of nonempty divisorial prime (semigroup) ideals.
The nonempty divisorial ideals of $H$ form a free abelian monoid 
with basis $\mathfrak X(H)$ with respect to the divisorial product; 
the nonempty divisorial fractional ideals form a free abelian group 
on the same basis.
Explicitly, every nonempty divisorial ideal $\mathfrak a$ of $H$ 
is uniquely expressible as a divisorial product of prime ideals
\[
\mathfrak a = \mathfrak p_1 \cdots_v \mathfrak p_r =
\left( 
\prod_{\mathfrak p \in \mathfrak X(H)} 
\mathfrak p^{\mathsf v_{\mathfrak p}(\mathfrak a)} 
\right)_v.
\]
The set 
$\{\, \fp \in \mathfrak X(H) : \mathsf v_{\fp}(\mathfrak a) > 0 \,\}
= \{\, \fp \in \mathfrak X(H) : \fp \supseteq \mathfrak a \,\}$ is 
the \emph{support} of $\mathfrak a$.
The class group $G=\mathcal C(H)$ of $H$ is the quotient of the group 
of nonempty divisorial fractional ideals of $H$ by the subgroup of 
principal fractional ideals.
Let $[\mathfrak a] \in G$ denote the class of $\mathfrak a$.
In factorization theory, the set 
$G_0 = \{\, [\mathfrak p] : \mathfrak p \in \mathfrak X(H) \,\}$ 
of classes containing prime divisors is of central importance.

We have two main examples of Krull monoids in mind: the first are 
the Dedekind domains, which are Krull domains. In fact, Dedekind 
domains (apart from fields, which usually count as Dedekind domains, too)
are precisely the one-dimensional Krull 
domains~\cite[Thm.~2.10.6]{GeHa2006:NonUniq}.
In a Dedekind domain $D$, the map 
$\mathfrak a \mapsto \mathfrak a \setminus \{0\}$ is a bijection 
between ring ideals of $D$ and divisorial semigroup ideals of 
$D \setminus \{0\}$.

The divisorial product is the usual product of ideals, and the 
class group is the usual one (the group of fractional ideals by 
principal fractional ideals).

The following is another class of Krull monoids.
\begin{definition}
Let $G$ be an additively written abelian group and $G_{0} \subseteq G$
a nonempty subset. Let $\mathcal{F}(G_{0})$ be the free abelian
monoid with basis $G_{0}$.
\begin{enumerate}
\item\label{block}
The elements of $\mathcal{F}(G_{0})$ are called \emph{sequences}
over $G_0$ and are of the form
 \[S = \prod_{g \in G_{0}}g^{n_g},\]
where $n_g = \mathsf{v}_g(S) \in \N \cup \{0\}$ with $n_g = 0$ for
almost all $g \in G_0$.
\item
The \emph{length} of a sequence $S$ is
\[\card{S} = \sum_{g \in G_{0}}\mathsf{v}_{g}(S) ~~\in ~\N \cup \{0\}\]
and the \emph{sum} of $S$ is
\[\sigma(S)= \sum_{g \in G_{0}}\mathsf{v}_{g}(S)g ~~\in ~G. \]
The \emph{support} of $S$ is
\[
\supp(S) = \{\, g \in G_0 : \mathsf{v}_g(S) > 0\,\}.
\]
\item
The monoid
\[\mathcal{B}(G_{0}) =
 \bigl\{S \in \mathcal{F}(G_{0})~ :~ \sigma(S) =0 \bigr\}\]
is called the \emph{monoid of zero-sum sequences} over $G_{0}$ or the
\emph{block monoid}.
\end{enumerate}
\end{definition}

The irreducibles of $\mathcal B(G_0)$ are the 
\emph{minimal zero-sum sequences}:
non\-empty sequences whose sum is $0$, but which do not contain a 
nonempty proper subsequence whose sum is $0$.

The following key theorem links the factorization theory of Krull
monoids to that of monoids of zero-sum sequences, showing that the 
latter serve as a combinatorial model for the factorization in Krull 
monoids and, in particular, Dedekind domains.
A proof can be found in \cite[Theorem 3.4.10.1]{GeHa2006:NonUniq}.
For more expository accounts of this theory see 
\cite[Theorem 1.3.4.2]{Geroldinger09}, the survey 
\cite{GeroldingerZhong20}, or the expository article~\cite{BaginskiChapman11}.

\begin{theorem} \label{t:krull-transfer}
Let $H$ be a Krull monoid with class group $G$, and let 
$G_0 \subseteq G$ be the set of classes containing prime divisors.
Then there exists a transfer homomorphism
\[
\theta \colon H \to \mathcal B(G_0),\quad  a \mapsto [\fp_1] \cdots [\fp_r],
\]
where $aH = \fp_1 \cdots_v \fp_r$ with $\fp_i \in \mathfrak X(H)$.
\end{theorem}

The homomorphism $\theta$ is called the \emph{block homomorphism} of $H$.

%% file: PAN_pmm_NA_19_08_2025.tex
\section{Rings whose irreducible elements are all non-absolutely irreducible}%
\label{section:NA}
Examples of atomic domains all of whose irreducibles are 
non-absolutely irreducible occur among generalised power series rings.

\begin{remark}\label{rem:NM}
Let $K_1 \subseteq K_2$ be fields, $n \in \N$, and $R = K_1 + X^nK_2[[X]]$.
Let $H$ be the multiplicative monoid $R \setminus \{0\}$ and $M_n$ be the 
numerical monoid
$\{0\} \cup (n +\N_0)$. Then the monoid homomorphism
\begin{align*}
 \theta : & H \longrightarrow M_n   \\
  & uX^\ell \longmapsto \ell
\end{align*}
is a transfer homomorphism, 
where $u$ is a unit of $K_2[[X]]$ and $\ell \in M_n$.
\end{remark}

\begin{proposition}\label{Prop:non-abs}\label{pmm}
Let $K_1 \subseteq K_2$ be fields and $n \in \N$, and set 
$R = K_1 + X^nK_2[[X]]$.
Then the following are equivalent.
\begin{enumerate}
\item $K_1 = K_2$ and $n=1$.\label{Prop:non-abs-itemI}
\item $R$ is a UFD. \label{Prop:non-abs-itemII}
\item every irreducible of $R$ is prime. \label{Prop:non-abs-itemIII}
\item $R$ has a prime element. \label{Prop:non-abs-itemIV}
\item $R$ has an absolutely irreducible element.
\label{Prop:non-abs-itemV}
\end{enumerate}
\end{proposition}

\begin{proof}
The implications \ref{Prop:non-abs-itemI} $\Longrightarrow$ \ref{Prop:non-abs-itemII} 
$\Longrightarrow$ \ref{Prop:non-abs-itemIII} $\Longrightarrow$ \ref{Prop:non-abs-itemIV} 
$\Longrightarrow$ \ref{Prop:non-abs-itemV} follow immediately.

For \ref{Prop:non-abs-itemV} $\Longrightarrow$ \ref{Prop:non-abs-itemI}, 
suppose $K_1 = K_2$ and $n >1$. Then the units of $R$ are
the elements of $R$ with constant term in $K_2 \setminus \{0\}$. 
It follows from Remark~\ref{rem:NM} and Fact~\ref{fact:TH}
that the irreducible elements of $R$ are the polynomials of the form 
$r = uX^m$, where $u$ is a unit of $K_2[[X]]$ and $n \leq m \leq 2n-1$. 
Every irreducible of the form $r$ is not
absolutely irreducible since for any $n \leq t \leq 2n-1$, with $t \neq m,$
\begin{equation*}
r^t = u^tX^t \cdot \underbrace{X^t \cdots X^t}_{m-1 \text{ copies}}
\end{equation*}
is a factorization of $r$ essentially different from
\[
  \underbrace{r \cdots r}_{t \text{ copies}}.
\]

 Suppose $K_1 \neq K_2$. Then the units of $R$ are the elements of $R$
 with constant term in $K_1 \setminus \{0\}$. Similarly, the
 irreducible elements of $R$ are the polynomials of the form $uX^m$,
  where $u$ is a unit of
$K_2[[X]]$ and $n \leq m \leq 2n-1$. Moreover, if $u_1, u_2$ are units of 
$K_2[[X]]$, then $u_1X^m \not \sim u_2X^m$ if
$u_1u^{-1}_2$ has a constant
term in $K_2 \setminus K_1$. It follows that each irreducible of the
form $uX^m$ is not absolutely irreducible since
\begin{equation*}
(uX^m)^2 = ucX^m \cdot uc^{-1}X^m
\end{equation*}
is a factorization of $(uX^m)^2$ essentially different from $uX^m \cdot uX^m$,
where $c \in K_2 \setminus K_1$.
\end{proof}

%% file: PAN_ppm_ANANP_19_08_2025.tex
\section{Rings with both absolutely and non-absolutely
irreducible elements, but no primes}\label{section:ANANP}

As an example of an atomic domain that has no prime element and
contains both absolutely and non-absolutely irreducible
elements, we propose the ring of integer-valued polynomials $\Int(\Z)$,
or, more generally, $\Int(\cO_K)$, where $\cO_K$ is the ring of integers
in a number field $K$. For a domain $D$ with quotient field $K$, the
ring of integer-valued polynomials on $D$ is
\[\Int(D) = \{f \in K[x] \mid f(D) \subseteq D \}.\]

\begin{definition}
The \emph{fixed divisor} of $f \in \Int(D)$, abbreviated $\fd(f)$, is 
the ideal of $D$ generated by $f(D)$. If the fixed divisor is a principal
ideal, we say $\fd(f)=c$, by abuse of notation, for $\fd(f)=(c)$.
A polynomial $f \in \Int(D)$ with $\fd(f)=1$ is called
\emph{image-primitive}.
\end{definition}

It is clear that $\fd(f)\cdot \fd(g)\supseteq \fd(fg)$, but note that
the inclusion may be strict. In any case, all divisors in $\Int(D)$ of
an image-primitive polynomial $f \in \Int(D)$ are image-primitive.

\subsection{No prime elements}\label{IntZnoprimes}
Anderson, Cahen, Chapman, and Smith~%
\cite{AndersonS-Cahen-Chapman-Smith:1995:fac-iv}
showed that $\Int(\Z)$ has no prime element by using the fact that
$\Int(\Z)$ is a Pr\"ufer domain whose maximal ideals are known and
are not principal. They argue that a Pr\"ufer domain never has any
principal prime ideals other than $(0)$ and, possibly, maximal ideals.
At the same time, no maximal ideal of $\Int(\Z)$ is principal (or even
finitely generated).  Their argument readily generalizes to $\Int(\cO_K)$.

We will here give an elementary, more explicit, proof that $\Int(\Z)$
(and, more generally, $\Int(\cO_K)$) has no prime element, by exhibiting,
for every potential prime element $p$, a product $ab$ such that $p$
divides $ab$, but $p$ divides neither $a$ nor $b$.

The only non-trivial fact needed is that every non-constant
polynomial in $\Z[x]$ has zeros modulo infinitely many primes, and,
more generally, for every number field $K$, every non-constant
polynomial in $\cO_K[x]$ has zeros modulo infinitely many maximal ideals
of $\cO_K$.

To see this we refer to a few facts about d-rings, a notion introduced
independently by Brizolis~\cite{brizolis1975:hilbert},
and Gunji and McQuillan~\cite{Gunji-Donald1976:divi}.

\begin{definition}
A domain $D$ is a {\em d-ring\/} if for every
non-constant polynomial $f\in D[x]$ there exists a maximal ideal
$M$ of $D$ and an element $d\in D$ such that $f(d)\in M$.
\end{definition}

So, $D$ being a d-ring just means that a polynomial $f\in D[x]$
cannot map $D$ into the set of units of $D$ unless $f$ is a constant.
It is easy to see that $\Z$ is a d-ring. Indeed, any $f\in \Z[x]$
such that $f(\Z)\subseteq \{1, -1\}$ must be constant.

Alternatively, d-rings can be characterized as those domains for which
every integer-valued rational function is an integer-valued polynomial.
We summarize what we need to know about
d-rings (cf.~\cite{CaCh1997:IVP}, \S VII.2).

\begin{fact}\cite[Lemma 1.3]{brizolis1975:hilbert},
\cite[Prop.~1]{Gunji-Donald1976:divi}.\label{d-ring-equiv}
The following are equivalent:
\begin{enumerate}
\item $D$ is a d-ring.
\item For every non-constant $f\in D[x]$, the intersection of the
maximal ideals $M$ of $D$ for which $f$ has a zero modulo $M$ is $(0)$.
\item For every non-constant $f\in \Int(D)$, there exists a
maximal ideal $M$ of $D$ and an element $d\in D$ such that $f(d)\in M$.
\end{enumerate}
\end{fact}

We conclude from Fact~\ref{d-ring-equiv} that every non-constant
polynomial in $\Z[x]$ has zeros modulo infinitely many primes.

\begin{lemma}[Anderson, Cahen, Chapman, and Smith~%
\cite{AndersonS-Cahen-Chapman-Smith:1995:fac-iv}]
$\Int(\Z)$ has no prime element.
\end{lemma}
\begin{proof}
First, no constant can be prime, because, if $p \in \Z$ is a
non-zero non-unit, then $p$ divides $(x-r_1) \cdots (x-r_p)$, where
$r_1, \ldots, r_p$ is a complete set of residues modulo $p\Z,$
but, since $(x-r_i)/p$ is not integer-valued, $p$ does not divide
any individual linear factor.

Now consider $G \in \Int(\Z)$ non-constant, $G = \frac{g}{d}$
with $g(x) \in \Z[x]$ and $d \in \Z$.  Let $p\in \Z$ be prime such
that $g$ has a zero modulo $p$ but the polynomial function induced
by $g$ is not constant zero modulo $p$.
(Such a prime $p$ exists because $g$ has a zero modulo infinitely
many primes, but only finitely many primes divide the fixed divisor
of $g$.)

Since $g\in \Z[x]$, the residue class of $g(r)$ modulo $p$ depends
only on the residue class of $r$ modulo $p$.
Let $r_1,\ldots, r_k$ be a complete set of representatives of those
residue classes modulo $p$ on which $g$ takes a non-zero value
modulo $p$. Let $h(x)=\prod_{i=1}^k (x-r_i)$. Note that $p\dividesnot h(c)$
when $c$ is a zero modulo $p$ of $g$ (such as exist by assumption), so that
$\frac{h(x)}{p}$ is not integer-valued.

Then, both $\frac{h(x)G(x)}{p}$ and $\frac{(G(x)+p)h(x)}{p}$ are
integer-valued. This means that $G(x)$ divides $\frac{(G(x)+p)h(x)G(x)}{p}$,
but $G(x)$ divides neither $G(x)+p$ (not even in $\Q[x]$) nor
$\frac{h(x)G(x)}{p}$ (because $\frac{h(x)}{p}$ is not integer-valued).

\end{proof}

To generalize to $\Int(\cO_K)$, we note that $\cO_K$ is a
d-ring for every number field $K$:

\begin{fact}\label{Fact:d-ring}
\cite[Prop.~3, Corollary~2]{Gunji-Donald1976:divi}.
Let $D\subseteq R$ be domains, and $D$ a d-ring.
\begin{enumerate}
\item If $R$ is integral over $D$, then $R$ is a d-ring.
\item If $R$ is finitely generated as a ring over $D$, then $R$
is a d-ring.
\end{enumerate}
\end{fact}

Since $\Z$ is a d-ring, it follows by Fact~\ref{Fact:d-ring} that
$\cO_K$, too, is a d-ring.

\begin{lemma}\label{IntOKnoprimes}
Let $\cO_K$ be the ring of integers
in a number field $K$. Then $\Int(\cO_K)$ has no prime element.
\end{lemma}
\begin{proof}
First, no constant can be prime, because, if $p \in \cO_K$ is a
non-zero non-unit, then $p$ divides $(x-r_1) \cdots (x-r_k)$, where
$r_1, \ldots, r_k$ is a complete set of residues modulo $p\cO_K,$
but $p$ does not divide any individual linear factor.

Now consider $G \in \Int(\cO_K)$ non-constant, $G(x) = \frac{g}{d}$
with $g(x) \in \cO_K[x]$ and $d \in \cO_K$. Because there are only
finitely many ramified primes and only finitely many primes
dividing the fixed divisor of $g$, there exists an unramified prime
$p \in \Z$, say  $p\cO_K = P_1 \cdot \ldots \cdot P_r$, such that
\begin{enumerate}
\item $g$ has a zero modulo $P_1$,
\item no $P_i$ divides the fixed divisor of $g$.
\end{enumerate}
Let $k$ be the maximal number of different residue classes modulo
any one $P_j$ on which $g$ assumes a non-zero value modulo $P_j$.
Choose $R = \{r_1, \ldots, r_k\} \subseteq \cO_K$ such that
\begin{enumerate}
\item for each $P_j$, $R$ contains a complete
set of representatives of the residue classes modulo $P_j$ on
which $g$ assumes a non-zero value modulo $P_j$
\item $g(r_i)$ is not zero modulo $P_1$ for any $r_i \in R$.
\end{enumerate}
Let $h = \prod_{i=1}^k(x-r_i)$, then
\[G \divides \frac{g(x)h(x)(g(x)+dp)}{dp} = G(x)\frac{h(x)(g(x)+dp)}{p},
\text{ but } \]
\[G \dividesnot \frac{g(x)h(x)}{dp}  \text{ and } G \nmid (g(x)+dp).\]
\end{proof}

\subsection{Absolutely irreducible elements}
Examples of absolutely irreducible elements of $\Int(\Z)$ include the
binomial polynomials
\[\binom xn= \frac{x(x-1)(x-2) \cdots (x-n+1)}{n!}\]
for $n >1$. If $n$ is prime, this is elementary, as
already McClain~\cite{McClain2004:thesis} remarked in her honor's thesis.
For general $n$, it is non-trivial and was shown by
Rissner and the fifth author~\cite{Rissner-Windisch2021:binom}.
Their result has been generalized to function fields by Tichy and the
fifth author~\cite{Tichy-Windisch2024:Carlitz}, but we are here
concerned with rings of integer-valued polynomials on number fields,
where we can provide quite elementary examples of absolutely
irreducible elements as follows:

\begin{lemma}\label{IntOKabsirred}
For any number field $K$, there exist absolutely irreducible
elements in $\Int(\cO_K)$.
\end{lemma}

\begin{proof}
Let $p \in \cO_K$ be an irreducible element with square-free
factorization into prime ideals; $p\cO_K = P_1 \cdot \ldots \cdot P_r$.
(Such an element exists because there are unramified primes and $\cO_K$
is atomic.)
Then let $q = \max_{1 \leq j \leq r}[\cO_K:P_j]$. W.l.o.g., $[\cO_K:P_1]=q$.
Let $r_1, \ldots, r_q$ be a complete system of residues modulo $P_1$,
containing one (but not more than one) complete system of residues
modulo $P_i$ for each $i$ and not containing a complete system of
residues modulo any other primes.
Set
\[ f(x) = (x-r_1) \cdots (x-r_q) \quad\text{ and }\quad
F(x) = \frac{f(x)}{p}.\]
We will show that $F$ is absolutely irreducible.

Suppose $F^m$ factors as $F^m = G_1\ldots G_s$, with each $G_i$ 
irreducible in $\Int(\cO_K)$, and $G_i=c_ig_i$ with $c_i\in K$, 
$g_i$ monic in $K[x]$.
Since $F^m$ is image-primitive, so is $G_i$ for each $i$.
This means that $c_i=(\fd(g_i))^{-1}$.

Let $v$ be the essential valuation corresponding to $P_1$, normalized
to have value group $\Z$. Then, in particular, $v(c_i) = - v(\fd(g_i))$,
therefore,
\[
\sum_{i=1}^s v(c_i) = v(p^{-m})=-m 
\quad\text{\ and\ }\quad
\sum_{i=1}^s v(\fd(g_i)) = m.
\]

This can only happen if each $g_i$ is a power of $f$; $g_i= f^{m_i}$ with 
$\sum_{i=1}^s m_i = m$.

To see this, consider that each $g_i$ is a product of monic linear 
polynomials $(x-r_j)$, $1\le j\le q$, where $r_1,\ldots, r_q$ form a
complete system of residues modulo $P_1$. For such a polynomial
$g=\prod_{j=1}^q(x-r_j)^{k_j}$, clearly $v(\fd(g))=\min_{1\le j\le q} k_j$.

Returning to $F^m = c_1g_1\ldots c_sg_s$, where 
$g_i=\prod_{j=1}^q(x-r_j)^{k_{ij}}$ then, of course, $\sum_{i=1}^s k_{ij}=m$
for each $j$, while, on the other hand, 
\[
m= \sum_{i=1}^s v(\fd(g_i)) = \sum_{i=1}^s \min_{1\le j\le q} k_{ij}.
\]
This implies that for each $1\le i\le s$ and $1\le h\le q$, necessarily
$k_{ih} = \min_{1\le j\le q} k_{ij}$, so that each linear factor occurs
in $g_i$ to the same exponent. If $m_i=\min_{1\le j\le q} k_{ij}$
then $g_i=f^{m_i}$.

Now $c_i=(\fd(g_i))^{-1}=p^{-m_i}$, and, therefore,
$G_i=c_ig_i=f^{m_i}p^{-m_i}= F^{m_i}$. As $c_ig_i=G_i$ was assumed irreducible,
$m_i=1$ follows.
\end{proof}

There are many other examples of absolutely irreducible elements in
$\Int(\cO_K)$, or more generally in $\Int(D),$ where $D$ is a
Dedekind domain with at least one finite residue field and
torsion class group~\cite[Corollary 8.9.]{Frisch-Nakato-Rissner2022:split}.

\subsection{Non-absolutely irreducible elements}
As an example of a non-absolutely irreducible element of $\Int(\Z)$,
consider
\[f = \frac{x(x^2+3)}{2}, \text{ noting that } f^2 =
 \frac{x^2(x^2+3)}{4} \cdot (x^2+3).\]
This example, taken from the third author's paper~\cite{SN2020:NonAbs}
on non-absolutely irreducible integer-valued polynomials,
generalises to $\Int(\cO_K)$ as follows.

\begin{lemma}\label{IntOKnonabs}
For any number field $K$, there exist
non-absolutely irreducible elements in $\Int(\cO_K)$.
\end{lemma}

\begin{proof}
Let $p \in \cO_K$ be an irreducible element with square-free
factorization into prime ideals; $p\cO_K = P_1 \cdot \ldots \cdot P_r$.
(Such an element exists because there are unramified primes and $\cO_K$
is atomic.) Then let $q = \max_{1 \leq j \leq r}[\cO_K:P_j]$. W.l.o.g., $[\cO_K:P_1]=q$.
Let $r_1, \ldots, r_q$ be a complete system of residues modulo $P_1$,
containing one (but not more than one) complete system of residues
modulo $P_i$ for each $i$ and not containing a complete system of
residues modulo any other primes.
Set $g(x) = (x-r_1)^2$ and $h(x) = (x-r_2) \cdots (x-r_q)$.

Let $H, G$ of the same degree as $g$ and $h$, respectively,
be irreducible in $K[x]$ and non-associated in $K[x]$,
such that for any product of copies of $g$ and $h$, the fixed divisor
is the same as that of any modified product in which some copies of
$g$ have been replaced by $G$ and some copies of $h$ by $H$.
(That such $G$ and $H$ exist has been shown by some of the present
authors together with
R.~Rissner~\cite[Lemma 3.3]{Frisch-Nakato-Rissner2019:Sets-of-lengths}.)

Let \[F(x) = \frac{G(x)H(x)}{p}.\]
Then $F$ is irreducible in $\Int(\cO_K)$, but not
absolutely irreducible, because
\[F^2 = \frac{G(x)H(x)^2}{p^2} \cdot G(x).\]
\end{proof}

Regarding non-absolutely irreducible elements of $\Int(\cO_K)$,
we can likewise generalize examples where the $n$-th power of an
irreducible element has factorizations of length other than $n$
(for instance, \cite[Example 4.4]{SN2020:NonAbs}) from $\Int(\Z)$
to $\Int(\cO_K)$ by using
\cite[Lemma 3.3]{Frisch-Nakato-Rissner2019:Sets-of-lengths} as in
the above proof. To summarize:

\begin{theorem}\label{ppm}
For any number field $K$, the ring of integer-valued polynomials
on algebraic integers,
\[\Int(\cO_K) = \{ f\in K[x] \mid f(\cO_K)\subseteq \cO_K\},\]
is a ring without prime elements containing both absolutely
irreducible and non-absolutely irreducible elements.
\end{theorem}

\begin{proof}
The non-existence of primes is shown in Lemma~\ref{IntOKnoprimes};
the existence of absloutely irreducible elements in 
Lemma~\ref{IntOKabsirred}, and the existence of non-absolutely irreducible
elements in Lemma~\ref{IntOKnonabs}.
\end{proof}

%% file: PAN_pmp_NAPNoA_19_08_2025.tex
\section{Rings with non-absolutely irreducible elements and primes,
but no other absolutely irreducible elements}\label{section:NAPNoA}
Examples of atomic domains that have no absolutely irreducible elements,
but contain both prime elements and non-absolutely irreducibles arise for 
instance from the D+M construction~\cite{Anderson-Zafrullah1991:rings}.
\begin{example}\label{pmp}
Let $R = K_1 + xK_2[x]$, where $K_1 \subsetneq K_2$ are fields. Then
$R$ is atomic and its irreducible elements are of the form
\begin{enumerate}
\item $ax$, where $a \in K_2$ or
\item $a(1 + xf(x))$, where $a \in K_1$, $f(x) \in K_2[x]$,
 and $1 + xf(x)$ is irreducible in $K_2[x]$, 
see~\cite{Anderson-Zafrullah1991:rings}.
\end{enumerate}
Furthermore, every irreducible of the form 
$a(1 + xf(x))$ is prime~\cite{Anderson-Zafrullah1991:rings}.
The irreducible elements of the form $ax$ are not absolutely irreducible,
because
 \begin{equation*}
(ax)^2 = acx \cdot ac^{-1}x
\end{equation*}
with $c \in K_2 \setminus K_1$.
\end{example}

With a bit more effort, we can even construct Noetherian semilocal integral domains of this type.

\begin{proposition}\label{prop:intersection}\label{pmpOne}
Let $D$ and $V$ be one-dimensional Noetherian local integral domains
that have a common quotient field $L$, and set $R = D \cap V$. 
Suppose further that the following hold:
\begin{enumerate}
\item\label{a}
$V$ is a discrete rank one valuation ring of $L$.
\item\label{b}
There exists a prime element $\pi$ of $V$ that is also a prime element of $R$.
\footnote{Of course, the prime element $\pi$ of $V$ is unique up to associativity \emph{in $V$}.
In the course of the proof we will see $\pi V \cap R = \pi R$,  so all associates of $\pi$ in $V$ that are contained in $R$ are also associated to $\pi$ \emph{in} $R$.}
\item\label{c} 
All the valuation rings that appear as localizations of the 
integral closure of $D$ are distinct from $V$.
\item\label{d}
 $D$ is not a unique factorization domain.
\item\label{e}
There are at least two non-associated irreducible elements of $D$ that are 
also irreducible in $R$.
\end{enumerate}

Then $R$ is a one-dimensional Noetherian integral domain \textup(so, in
particular, an atomic domain\textup) that has exactly two maximal ideals 
and precisely one prime element up to associativity, namely $\pi$. In addition, it has 
irreducible elements distinct from $\pi$ and all of these are 
not absolutely irreducible.
\end{proposition}

\begin{proof}
First, for the sake of completeness, we want to recall the full 
argument why $D$ cannot be contained in $V$ (which will be needed 
later in the proof). Assume to the contrary that $D \subseteq V$.
Since $V$ is integrally closed, the integral closure $\overline{D}$
of $D$ is also contained in $V$.

Now there are two cases. In case $\pi V \cap \overline{D} = (0)$, 
all  non-zero elements of $\overline{D} \setminus \{0\}$ are invertible
in $V$ and hence 
$L = (\overline{D} \setminus \{0\})^{-1} \overline{D} \subseteq V$,
which is a contradiction. 
Otherwise, $\pi V$ lies over a maximal ideal of $\overline{D}$ and 
hence $\overline{D}_{\pi V \cap \overline{D}} = V$. This is in 
contradiction to the assumption \ref{c}. So, in total, we infer 
that $D$ is not contained in $V$.

Denote by $M$ the maximal ideal of $D$.
We are now in the situation 
of~\cite[Theorem 3]{Prekowitz1973:Intersections}, which implies that 
$D = R_{M \cap R}$ and $V = R_{\pi V \cap R}$.
Note that the prime ideals $M \cap R$ and $\pi V \cap R$ are not 
comparable by set-theoretic inclusion because, otherwise, $D$ or $V$
could not be a one-dimensional ring.
So, we can apply~\cite[Theorem 105]{Kaplansky1970:commutative} and 
get that $M \cap R$ and $\pi V \cap R$ are exactly the maximal ideals 
of $R$. In particular, using the prime ideal correspondence under 
localization, $R$ is a one-dimensional domain.

Because $R$ is one-dimensional and $\pi$ is a prime element of $R$ by assumption, we must have $\pi R = \pi V \cap R$.
In particular, $\pi \not \in M$, and so $\pi$ is a unit of $D$ but a non-unit of $V$. Thus we are able to 
use~\cite[Corollary 1.20]{Heinzer-Ohm:Noetherian} and infer that 
$R$ is Noetherian (and therefore atomic).

It is now left to show that the irreducible elements of $R$ not associated to 
$\pi$ are not absolutely irreducible. Since $\pi$ is prime in $R$,
all these irreducibles lie in $(M\cap R) \setminus (\pi V \cap R)$.
Now, \ref{e} says that there are at least two of them and hence, by 
\cite[Lemma 2.1]{ChKr2012:Atomic-decay} they cannot be absolutely 
irreducible.
\end{proof}

We now give concrete examples of integral domains $D$ and $V$ as 
in~\ref{prop:intersection}. 
In particular, we thus find another atomic integral domain that has
prime elements and irreducible elements that are not absolutely 
irreducible, but has no absolutely irreducible elements that are not prime.

\begin{example}
Let $K$ be any field in which $-1$ is not a square, for instance
$K = \mathbb Q$, and let $X$ be an indeterminate over $K$. Define 
$D$ and $V$ as the localizations
\[
D = K[X^2,X^3]_{(X^2,X^3)}, \ V = K[X]_{(X^2+1)},
\]
and set $R = D \cap V$.

Note that $(X^2+1)$ is indeed a prime ideal of $K[X]$ since $-1$ is 
not a square in $K$.
Both, $D$ and $V$, are one-dimensional Noetherian local integral domains
and their quotient field is just the function field $K(X)$.
In order to see this for $D$, just note that $K[X^2, X^3]$ is a 
numerical semigroup algebra and hence one-dimensional and Noetherian.

Moreover, $V$ clearly is a discrete rank one valuation ring with prime 
element $\pi = X^2 +1$. The integral closure of $D$ is the valuation ring 
$K[X]_{(X)}$ that is indeed distinct from $V$.

Next we show that $X^2+1$ is also a prime element of $R$.
For this, we can just argue that it generates the prime ideal
$P = \mathfrak{m}_V \cap R$, where $\mathfrak{m}_V$ 
denotes the unique maximal ideal of $V$.

Let $g \in P$. As $g$ is, a fortiori, an element of $D$, we can write 
it in the form $g = \frac{h}{s}$, where $h,s \in K[X^2,X^3]$ and 
$s \notin (X^2,X^3)K[X^2,X^3]$. On the other hand, $g$ is also in
$\mathfrak{m}_V$ and we can therefore write it as 
$g = \frac{a}{b}(X^2+1)$ with $a, b \in K[X]$ and $b \notin (X^2 +1)K[X]$.

It is our goal to show that $X^2 + 1$ divides $g$ in $R$. 
Since $\frac{a}{b} \in V$ by its choice, it suffices to prove that
$\frac{a}{b} \in D$. In order to do this, we clear denominators in 
the equation $\frac{h}{s} = \frac{a}{b} (X^2 +1)$ and arrive at
\[
bh = sa (X^2 +1)
\]
that we can view as an identity in $K[X]$. 
Since $X^2 +1 $ does not divide $b$ in $K[X]$, it has to divide $h$ 
and the unique cofactor is $\frac{sa}{b} \in K[X]$. 
Since $h \in K[X^2,X^3]$ has no linear term, it follows from
$h = \frac{sa}{b} (X^2 + 1)$ that the polynomial $\frac{sa}{b}$ 
also has no linear term, that is, $\frac{sa}{b} \in K[X^2,X^3]$. 
By choice, $s \in K[X^2,X^3] \setminus (X^2,X^3)K[X^2,X^3]$ and 
hence $\frac{a}{b} = \frac{1}{s} \cdot \frac{sa}{b} \in D$.

As a last step, we argue that the elements $X^2$ and $X^3$ are still 
irreducible in $R$. They are both elements of 
$R \cap (X^2,X^3)K[X^2,X^3]$ and, therefore, non-units. 
In the following, for an irreducible polynomial $p$ of $K[X]$, we 
denote by $\mathsf{v}_p$ the $p$-adic valuation on $K(X)$.

We carry out the argument for $X^2$; it is then analogous for $X^3$. 
So, decompose $X^2 = f \cdot g$ where $f$,~$g \in R$. We want to show that 
either $f$ or $g$ is a unit of $R$. The elements of $D=K[X^2,X^3]_{(X^2,X^3)}$
(and therefore those of $R$) either have $X$-adic valuation $0$ or 
$\geq 2$. 
Since $\mathsf{v}_X(f) + \mathsf{v}_X(g) = \mathsf{v}_X(X^2) = 2$, 
we can assume without loss of generality that $\mathsf{v}_X(f) = 2$ 
and $\mathsf{v}_X(g) =0$.

Furthermore, as elements of $V=K[X]_{(X^2+1)}$, the $(X^2+1)$-adic
valuation of $f$ and $g$ is $\geq 0$. Hence, the equality
$0 = \mathsf{v}_{X^2+1}(X^2) = \mathsf{v}_{X^2+1}(f) + \mathsf{v}_{X^2+1}(g)$
implies that $\mathsf{v}_{X^2+1}(g) = 0$. To conclude, $g$ is an element 
of $R$ that is in neither of the two maximal ideals of $R$ and therefore a unit.
This finishes the example.
\end{example}

%% file: PAN_ppp_PANA_21_08_2025.tex
\section{Rings with prime elements, absolutely irreducible elements that 
are not prime, and non-absolutely irreducible elements}\label{section:PANA}

For an atomic domain that has prime elements, absolutely irreducible 
elements that are not prime, and non-absolutely irreducible elements, 
we consider certain Krull domains and subrings of the ring of
integer-valued polynomials.

\begin{example}\label{pana}
Let $R = \Int(\Z) \subseteq \Q[x]$ be the ring of integer-valued 
polynomials as in section~\ref{section:ANANP} 
(or any atomic domain satisfying ACCP and having 
both absolutely irreducible and non-absolutely irreducible elements).

We show that $T = R[y]$, the polynomial ring in one indeterminate over 
$R$, is an atomic domain that has primes, as well as both absolutely and 
non-absolutely irreducible elements.

Since $R$ is an integral domain, $y$ is a prime element of $T$.

Again because $R$ is an integral domain, by degree considerations,
the set of polynomials of degree zero in $R[y]$ (which we may 
identify with $R\setminus\{0\}$) forms a saturated (i.e., divisor 
closed) multiplicatively closed subset of $R[y]$ containing all the 
units of $R[y]$.

A constant in $R[y]$, therefore, is a unit, or absolutely irreducible, 
or non-absolutely irreducible, in $R[y]$ if and only if it has the 
respective property in $R$. 

Since $R$ contains both absolutely and non-absolutely irreducible
elements, $R[y]$ contains both absolutely and non-absolutely 
irreducible elements (namely, the constants with the respective
property in $R$), as well as a prime element, $y$.

Also, since $R$ satisfies ACCP, so does $R[y]$, and $R[y]$ is,
therefore, atomic.
\end{example}

\begin{proposition}\label{pana1}
For each prime number $p$, the ring
\[R(p) = \left\{\,\frac{g}{p^n} \in \Int(\Z) 
\text{ with } g \in \Z[X] \text{ and } n \in \N_0 \,\right\}\]
has infinitely many prime elements. 
Furthermore, $R(p)$ has both absolutely irreducible
elements that are not prime and non-absolutely irreducible elements.
\end{proposition}

\begin{proof}
First, each prime number $q \ne p$ is prime in $R(p)$:
If $q$ divides a product $(g_1/p^{n_1}) \cdots (g_r/p^{n_r})$ in $R(p)$, 
then $q$ divides $g_1 \cdots g_r$ in $\Z[X]$.
Because $\Z[X]$ is a UFD with $q$ prime, without restriction $g_1=q f_1$ with $f_1 \in \Z[X]$.
For all $a \in \Z$, also $b \coloneqq (q f_1(a))/p^{n_1} \in \Z$, and since $q \neq p$, it follows that $ f_1(a)/p^{n_1} \in \Z$, showing $f_1 / p^{n_1} \in R(p)$.
We have therefore shown that $q$ divides $g_1/p^{n_1}$ in $R(p)$, so $q$ is prime in $R(p)$.

Now, let $r_1, \ldots, r_p$
be a complete system of residues modulo $p$ and set
\[f(x) = \frac{(x-r_1) \cdots (x-r_p)}{p}.\]
Then $f$ is absolutely irreducible in $R(p)$. 
The polynomial $f$ is not prime because $f$
divides $(x-r_1) \cdots (x-r_p)$
but it does not divide any individual linear factor.

Furthermore, let $a_1, \ldots, a_p, b_{p+1}, \ldots, b_{p^2}$ be
a complete system of residues modulo $p^2$ with 
$b_i \not \equiv 0 \Mod{p}$ for $p+1 \leq i \leq p^2$. 
 Let $c_1, c_2 \in \Z$ such that $c_1 \equiv c_2 \equiv 0 \Mod{p^2}$ 
and $c_1 \neq c_2$.
Set $g(x) = \prod_{k = p + 1}^{p^2}(x - b_k)$ and
\[f(x) = \frac{g(x)(x-c_1)(x-c_2)^{e-1}}{p^e},\]
where $e = \mathsf v_p(p^2!) = p+1$. 
Then $f$ is irreducible in $R(p)$, but not absolutely irreducible, because
\[f^2 =
 \frac{g(x)(x-c_1)^2(x-c_2)^{e-2}}{p^e} \cdot \frac{g(x)(x-c_2)^{e}}{p^e}\]
is a factorization of $f^2$ essentially different from $f \cdot f$.
\end{proof}

The non-absolutely irreducible elements of $R(p)$ whose $n$-th power 
has factorizations of different lengths can be constructed by adapting 
known examples in $\Int(\Z)$ \cite[Example 4.1 \& Example 4.4]{SN2020:NonAbs}.

\begin{fact} \label{f:algint}
Let $\mathcal{O}_K$ be the ring of integers of a number
field $K$. Then the following hold.
\begin{enumerate}
\item $\mathcal{O}_K$ has infinitely many prime elements.
\item If $\mathcal{O}_K$ is not a unique factorization domain, then
$\mathcal{O}_K$ has absolutely irreducible elements that are not prime,
and non-absolutely irreducible elements, see
\cite[Theorem 3.1]{ChKr2012:Atomic-decay} or Corollary~\ref{c:krull-monoid-all-absirred} below.
\end{enumerate}
\end{fact}

\begin{example}\label{pana2}
Let $d = -ec,$ where $e, c \in \Z, e, c \geq 2, d$ is square-free and $d \not\equiv 1 \Mod{4}.$ Set $R = \Z[\sqrt{d}]$. Firstly, the element $\sqrt{d} = \sqrt{-ec}$ is irreducible in $R$,
but not absolutely irreducible, because $(\sqrt{d})^2 = -e \cdot c$ is a
non-trivial factorization of $(\sqrt{d})^2$ (not necessarily into irreducibles).

Secondly, let $r = a + b\sqrt{d} \in R$ and $N(r) = a^2 - db^2$,
the norm of $r$. 
Then 2 is irreducible in $R$ because
$N(2) = 4$ and $N(r) \neq 2$ for all $r \in R$ (because $d \le -6$).
The element $2$ is not prime in $R$ because
\begin{enumerate}
\item 
if $d \equiv 2 \Mod{4}$, then $2 \dividess (\sqrt{d})^2$, but
$2 \dividenot \sqrt{d}$, and
\item
if $d \equiv 3 \Mod{4}$, then 
$2 \dividess (1 + \sqrt{d})(1 - \sqrt{d})$, but 2 does not
divide the individual factors.
\end{enumerate}

More generally, by \cite[Theorem 25]{Marcus1977NF}, the prime decomposition
of $2R$ is
\begin{enumerate}
\item $2R = (2, \sqrt{d})^2$ if $d \equiv 2 \Mod{4}$, and
\item $2R = (2, 1 + \sqrt{d})^2$ if $d \equiv 3 \Mod{4}$.
\end{enumerate}
It follows by \cite[Theorem 3.1]{ChKr2012:Atomic-decay} that $2$ is
absolutely irreducible in $R$.

Lastly, it follows by \cite[Theorem 25]{Marcus1977NF} that every odd prime
$p$ such that $p \dividenot d$ and 
$d^{\frac{p-1}{2}} \equiv -1 \Mod{p}$, is prime in $R$.
\end{example}


\section{Absolutely Irreducible Elements in Krull monoids}

The examples in Sections~\ref{section:NPA} and~\ref{section:PA} will be constructed using Dedekind domains.
The multiplicative monoid of non-zero elements of a Dedekind domain is a Krull monoid \cite[Chapter 2.10]{GeHa2006:NonUniq}. 
To lay the groundwork for Sections~\ref{section:NPA} and~\ref{section:PA}, we therefore first prove some results on absolute irreducible elements in Krull monoids, culminating in Theorem~\ref{t:krull-allai}, which forms the basis of Propositions~\ref{Prop:withoutprime} and~\ref{prop:exm-allai-but-not-ufd}.

As corollaries we also recover a generalization of a theorem of Chapman and Krause (Corollary~\ref{c:krull-monoid-all-absirred}) and a theorem of Angermüller (Corollary~\ref{c:angermueller}).
Corollary~\ref{c:angermueller} in particular rules out the existence of Krull monoids having non-prime irreducibles, but having every absolutely irreducible prime.
This shows that no example as in Section~\ref{section:NAPNoA} is possible among Krull monoids.
Corollary~\ref{c:krull-monoid-all-absirred} shows that, if every class of the class group of a Krull monoid $H$ contains a prime divisor and there exist no non-absolutely irreducible irreducibles, then $H$ is already factorial.
This explains why in Sections~\ref{section:NPA} and~\ref{section:PA} we have to construct domains in which the set of classes containing prime divisors is a proper subset of the entire class group.

In this section, let $H$ be a Krull monoid with class group $G$.
Let $G_0 \subseteq G$ denote the set of classes containing prime 
divisors, and let $G_1 \subseteq G_0$ denote the set of classes 
containing \emph{exactly} one prime divisor.

Absolutely irreducible elements in Krull monoids have been characterized 
in various ways.
The following extends \cite[Proposition 7.1.4]{GeHa2006:NonUniq}
(where $G=G_0$ is assumed) and parts of 
\cite[Proposition 4.7]{Grynkiewicz2022} 
(where $H=\cB(G_0)$ with $G$ torsion-free).
A similar characterization is given in \cite[Lemma 5]{Angermueller2020}.

\begin{proposition} \label{p:krull-ai}
Let $\emptyset \ne S \subseteq \mathfrak X(H)$ be a finite set.
The following statements are equivalent.
  \begin{enumerate}
\item
\label{krull-ai:ai} There exists an absolute irreducible 
$a \in H$ with $\supp(aH)=S$.
\item 
\label{krull-ai:min} The set $S$ is minimal in
 $\{\, \supp(bH) : b \in H\setminus H^\times \,\}$
  \item \label{krull-ai:minirred} 
The set $S$ is minimal in 
$\{\, \supp(bH) : b \in H \text{ is irreducible} \,\}$.
\item\label{krull-ai:zss}
The family $([\fp])_{\fp \in S}$ in $G$ is $\Z_{\ge 0}$-linearly 
dependent, and every proper subfamily is $\Z_{\ge 0}$-linearly independent.
\item\label{krull-ai:indep}
The family $([\fp])_{\fp \in S}$ in $G$ is $\Z_{\ge 0}$-linearly
dependent, and every proper subfamily is $\Z$-linearly independent.
  \end{enumerate}
In case the equivalent conditions hold, the absolutely irreducible 
element with support $S$ is uniquely determined up to associativity.
\end{proposition}

\begin{proof}
\ref{krull-ai:ai}$\,\Rightarrow\,$\ref{krull-ai:min}
If $\supp(bH) \subsetneq S$, then $b \mid a^n$  for some $n \ge 1$ and $b$ is not associated to $a$.

\ref{krull-ai:min}$\,\Rightarrow\,$\ref{krull-ai:minirred}
Trivial.

\ref{krull-ai:minirred}$\,\Rightarrow\,$\ref{krull-ai:min}
Suppose there exists $b \in H \setminus H^\times$ with 
$\supp(bH) \subsetneq S$.
Let $u \in H$ be an irreducible element dividing $b$.
Then $\supp(uH) \subsetneq S$.

\ref{krull-ai:min}$\,\Leftrightarrow\,$\ref{krull-ai:zss}
Statement \ref{krull-ai:zss} is just a explicit way of stating 
\ref{krull-ai:min}.

\ref{krull-ai:zss}$\,\Rightarrow\,$\ref{krull-ai:indep}
Fix a non-zero vector $(\alpha_\fp)_{\fp \in S} \in \Z_{\ge 0}^S$ 
such that $\sum_{\fp \in S} \alpha_\fp [\fp] =0$.
Suppose, for the sake of contradiction, that there exists $S' \subsetneq S$ and a non-zero 
vector $(\beta_\fp)_{\fp \in S'} \in \Z^{S'}$ such that 
$\sum_{\fp \in S'} \beta_\fp [\fp] = 0$.
By the minimality of $S$, there must exist $\fp \in S'$ with 
$\beta_\fp < 0$.
Let $\fq \in S'$ be such that $\beta_\fq < 0$ and so that 
$\alpha_\fq / \abs{\beta_\fq}$ is minimal among all
$\alpha_\fp / \abs{\beta_\fp}$ with $\beta_\fp < 0$.
Then
\[
\abs{\beta_\fq} \alpha_\fp + \alpha_\fq \beta_\fp \ge 0
\]
for all $\fp \in S$, and equality holds for $\fp=\fq$.
Since
\[
\abs{\beta_\fq} \sum_{\fp \in S} \alpha_\fp [\fp] +
 \alpha_\fq \sum_{\fp \in S'} \beta_\fp [\fp] = 0,
\]
this contradicts the $\Z_{\ge 0}$-linear independence of 
$([\fp])_{\fp \in S \setminus\{\fq\}}$.

 \ref{krull-ai:indep}$\,\Rightarrow\,$\ref{krull-ai:ai}
Consider the group homomorphism $\sigma\colon \Z^S \to G$ given by 
$\sigma( (\alpha_{\fp})_{\fp \in S}) = \sum_{\fp \in S} \alpha_\fp [\fp]$.
Our assumptions ensure that there exists some 
$(\alpha_{\fp})_{\fp \in S} \in \Z_{>0}^S \cap \ker(\sigma)$.
Moreover, the image of $\sigma$ contains a torsion-free subgroup of 
rank $\card{S} - 1$.
Thus $\ker(\sigma)$ is free of rank one.
Since we know that $\ker(\sigma)$ contains a vector with all 
positive coordinates, we can also choose a generator (as an abelian 
group) $(\beta_{\fp})_{\fp \in S}$ with all positive coordinates.
The zero-sum sequences with support $S$ correspond to elements of 
$\ker(\sigma)$ with positive entries, and they are all positive 
multiples of $(\beta_{\fp})_{\fp \in S}$.

Hence there is, up to associativity, a unique irreducible with support 
contained in $S$.
This irreducible is then necessarily absolutely irreducible.
\end{proof}

While the absolute irreducibility of elements does not lift along 
the transfer homomorphism $H \to \cB(G_0)$ in general (Remark \ref{l:th-absirred}), additional 
knowledge of the set $G_1$ nevertheless allows us to characterize 
when \emph{every} irreducible is absolutely irreducible using $\cB(G_0)$.

\begin{theorem} \label{t:krull-allai}
 Let $H$ be a Krull monoid with class group $G$, let $G_0$ be the set of classes containing
 prime divisors, and let $G_1$ be the set of classes containing precisely one prime divisor.
 The following are equivalent.
\begin{enumerate}
\item\label{krull-allai:allai}
Every irreducible of $H$ is absolutely irreducible.
\item\label{krull-allai:zss}
Every irreducible of $\cB(G_0)$ is absolutely irreducible, and 
for every irreducible $U \in \cB(G_0)$ and every 
$g \in G_0 \setminus G_1$ it holds that $\val_g(U) \le 1$.
\end{enumerate}
\end{theorem}

\begin{proof}
Let $\theta \colon H \to \cB(G_0)$ denote the block homomorphism.

\ref{krull-allai:allai}$\,\Rightarrow\,$\ref{krull-allai:zss}
Since $\theta$ is a transfer homomorphism, the monoid $\cB(G_0)$ 
also has the property that every irreducible is absolutely irreducible 
(by Lemma~\ref{th-absirred:forward}).

For the second property, suppose that there exists an irreducible 
$U \in \cB(G_0)$ that is of the form $U =g^2 T$ with $T \in \cF(G_0)$, 
and that there exist $\fp \ne \fq \in \mathfrak X(H)$ such that 
$[\fp]=[\fq]=g$.
We may assume $T=[\fr_1]\cdots [\fr_k]$ with $\fr_i \in \mathfrak X(H)$.
Let $\mathfrak a = \fr_1 \cdots_v  \fr_k$.
Then $\fp^2 \cdot_v \mathfrak a$ and $\fp \cdot_v \fq \cdot_v \mathfrak a$
are principal ideals, say $aH = \fp^2 \cdot_v \mathfrak a$ and 
$bH = \fp \cdot_v \fq \cdot_v \mathfrak a$ with $a$,~$b \in H$.
Since $\theta$ is a transfer homomorphism, the irreducibility of $U$ 
implies that of $a$ and $b$.
However $a \mid b^2$ and so $b$ is not absolutely irreducible, 
contradicting our assumption.

\ref{krull-allai:zss}$\,\Rightarrow\,$\ref{krull-allai:allai}
Let $a \in H$ be irreducible.
Then $aH$ has a unique factorization
  \[
    aH = \fp_1^{e_1}  \cdots_v \fp_k^{e_k} \cdot_v \fq_1 \cdots_v \fq_l,
  \]
with pairwise distinct prime ideals $\fp_1$, $\ldots\,$,~$\fp_k$ with $[\fp_i]\in G_1$ and exponents $e_i \ge 1$, and prime ideals $\fq_1$, $\ldots\,$,~$\fq_l$ with $[\fq_i] \in G_0 \setminus G_1$.
(Assumption \ref{krull-allai:zss} applied to $\theta(a)$ implies $[\fq_i] \ne [\fq_j]$ for $i \ne j$.)

Suppose that $b \in H$ is an irreducible with $b \mid a^n$ for some $n \ge 1$.
  Then $\theta(b) \mid \theta(a)^n$.
  Since $\theta(a)$ is absolutely irreducible, we get $\theta(a) = \theta(b)$.
  Since $[\fp_i] \in G_1$ for all $1 \le i \le k$, the image $\theta(b)$ 
fully determines the multiplicity of each $\fp_i$ in $bH$.
  Now
  \[
 bH = \fp_1^{e_1}  \cdots_v \fp_k^{e_k} \cdot_v \fq_{i_1} \cdots_v \fq_{i_m},
  \]
 for some $1 \le i_1 \le \cdots \le i_m \le l$.
 Since $[\fq_{i_r}] \in G_0 \setminus G_1$ for $1 \le r \le m$, assumption \ref{krull-allai:zss} applied to $\theta(b)$ implies $[\fq_{i_r}] \ne [\fq_{i_s}]$ for $r \ne s$.
 Thus $1 \le i_1 < \cdots < i_m \le l$. 
 Now $|\theta(a)|=|\theta(b)|$ shows $\{i_1,\ldots,i_m\}=\{1,\ldots,l\}$, and so $aH=bH$.
\end{proof}

\subsection{Consequences for Krull monoids} \label{subsec:krull-corollaries}

As stated at the beginning of the section, we now note some easy consequences of Theorem~\ref{t:krull-allai} that connect to the existing literature and are of interest in their own right.
Aside from the fact that Corollary~\ref{c:krull-monoid-all-absirred} can be used to obtain (ii) of Fact~\ref{f:algint}, these results are however not used in this paper.

Given Theorem~\ref{t:krull-allai}, and the fact that we are often interested 
in cases when $G_0=G$ and $G_1=\emptyset$, it is useful to determine 
when every irreducible of $\cB(G)$ is absolutely irreducible.
The equivalence of the first two statements of the following corollary 
is well-known.

\begin{corollary} \label{c:bg-all-absirred}
  Let $G$ be an abelian group.
Then the following are equivalent for the monoid of zero-sum sequences
$\mathcal B(G)$.
  \begin{enumerate}
  \item \label{baa:group} $\card{G} \le 2$.
  \item \label{baa:factorial} $\mathcal B(G)$ is factorial.
  \item \label{baa:absirred} Every irreducible of $\mathcal B(G)$
 is absolutely irreducible.
  \end{enumerate}
\end{corollary}

\begin{proof}
 \ref{baa:group}$\,\Rightarrow\,$\ref{baa:factorial}:
We recall the (well-known) argument.
If $G$ is trivial, then the sequence $0$ (that is, the sequence of 
length $1$ consisting of the single element $0 \in G$) is the only 
irreducible of $\cB(G)$, and $\cB(G) \cong \N_0$ is factorial.
If $G=\{0,g\} \cong \Z/2\Z$, then $0$ and $g^2$ are the only 
irreducibles of $\cB(G)$, and they are both prime, so that
$\cB(G) \cong \N_0^2$.

\ref{baa:factorial}$\,\Rightarrow\,$\ref{baa:absirred}: Trivial.

\ref{baa:absirred}$\,\Rightarrow\,$\ref{baa:group}: 
We prove the contrapositive.
  Assume $\card{G} \ge 3$.
  Then one of the following three cases must occur.
\begin{itemize}
 \item
\emph{$G$ contains an element $g$ of finite order $n \ge 3$.}
Consider the irreducibles $S=g^n$, $S'=(-g)^n$ and $T=g(-g) \in \cB(G)$.
Then $S S' = T^n$, and so $T$ is not absolutely irreducible (note $g \ne -g$).
\item 
\emph{$G$ contains two independent elements $g$, $h$ both of order $2$.}
Consider the irreducibles $S=g^2$, $S'=h^2$, $S''=(g+h)^2$, and 
$T=gh(g+h) \in \cB(G)$.  Then $SS'S'' = T^2$.

\item 
\emph{$G$ contains an element $g$ of infinite order.}
Consider $S=(3g)(-g)^3$, $S'=(3g)^2(-2g)^3$ and $T=(-g)(-2g)(3g) \in \cB(G)$.
 Then $S$, $S'$, and $T$ are irreducible and $SS'=T^3$. \qedhere
 \end{itemize}
\end{proof}

The statements of the previous lemma are 
further equivalent to $\cB(G)$ being half-factorial \cite[Theorem 3.4.11.5]{GeHa2006:NonUniq}.

The following (straightforwardly) generalizes the theorem of Chapman
 and Krause \cite[Corollary 3.2]{ChKr2012:Atomic-decay}, who proved 
the statement for $H=\mathcal \cO_K^\bullet$ with $\mathcal \cO_K$ 
a ring of algebraic integers in a number field.

\begin{corollary} \label{c:krull-monoid-all-absirred}
Let $H$ be a Krull monoid such that every class in the class group $G$ 
contains a prime divisor \textup(that is, $G=G_0$\textup).
Then every irreducible element of $H$ is absolutely irreducible if and 
only if $H$ is factorial.
\end{corollary}

\begin{proof}
If $H$ is factorial, then every irreducible is prime and hence 
absolutely irreducible.

For the converse, suppose that every irreducible of $H$ is 
absolutely irreducible. By Theorem~\ref{t:krull-allai}, every 
irreducible of $\cB(G)$ must be absolutely irreducible.
Thus $\card{G} \le 2$ by Corollary~\ref{c:bg-all-absirred}.
To show that $G$ is trivial it therefore suffices to show 
$G \not \cong \Z/2\Z$.

Suppose that $G =\{0,g\} \cong \Z/2\Z$.
By (b) of \cite[Theorem 2.5.4.1]{GeHa2006:NonUniq} there must 
exist at least two distinct non-zero divisorial prime ideals 
$\mathfrak p$ and $\mathfrak q$ of non-trivial class 
$[\mathfrak p]=[\mathfrak q]=g$.
Therefore $\mathfrak p^2$, $\mathfrak q^2$ and $\mathfrak {p \cdot_v q}$ 
are all principal and generated by irreducibles $u$, $v$, $w$, say, 
$\mathfrak p^2 = uH$, $\mathfrak q^2=vH$ and $\mathfrak {p \cdot_v q} = wH$.
  Now $w^2 \sim uv$ shows that $w$ is not absolutely irreducible.
\end{proof}

We will see below, in Proposition~\ref{prop:exm-allai-but-not-ufd}, 
that some assumption on $G_0$ is necessary for the previous proposition.
We also recover a theorem of Angerm\"uller; see 
\cite[Theorem 1(e)]{Angermueller2022} for a generalization to 
monadically Krull monoids.

\begin{corollary}[{\cite[Corollary 1(c)]{Angermueller2020}}] \label{c:angermueller}
A Krull monoid $H$ is factorial if and only if every absolutely 
irreducible element is prime.
\end{corollary}

\begin{proof}
 If $H$ is factorial, then every irreducible is prime.

For the converse, suppose that every absolutely irreducible element is prime.
Suppose that $H$ is not factorial.
Then there exist non-prime irreducibles.
Let $u\in H$ be a non-prime irreducible with $\supp(uH)$ minimal among 
all non-prime irreducibles.
Then $\supp(uH)$ is in fact minimal among all irreducibles: if $p \in H$ is a prime, then $pH$ is a divisorial prime ideal and hence $\supp(pH)=\{pH\} \subseteq \supp(uH)$ would imply that $pH$ appears in the factorization of $uH$ into divisorial prime ideals, meaning $p \mid u$, which is impossible since $u$ is a non-prime irreducible.
Then $u$ is absolutely irreducible by Proposition~\ref{p:krull-ai}, 
and therefore prime, a contradiction.
\end{proof}

\begin{remark}
Absolutely irreducible elements in Krull monoids have been studied in 
different settings and under different names.
For instance, in \cite{Grynkiewicz2022} they play a very central role 
in the setting of $H=\cB(G_0)$ with $G$ torsion-free, and are called 
\emph{elementary atoms}.
If $H$ is a normal affine monoid, then the absolutely irreducible elements 
correspond precisely to the extremal rays of the polyhedral convex cone, 
whereas the irreducible elements form the Hilbert basis of the monoid.
We refer to \cite{garciasanchez-krause-llena2023}, in particular to 
\S 4 and therein to Remarks 13 and 16 for a discussion of the terminologies.
\end{remark}

%% file: PAN_mpm_NPA_19_08_2025.tex
\section{Rings with irreducible elements that are all absolutely
irreducible, but none of them prime}
\label{section:NPA}

\begin{proposition}\label{Prop:withoutprime}\label{mpm}
There exists a Dedekind domain that is not half-factorial and such that all 
of its irreducible elements are absolutely irreducible but none of 
them prime.
\end{proposition}
\begin{proof}
Let $n \ge 2$ be an integer, let $G = \Z^n$, and let $\{e_1, e_2, \ldots, e_n \}$ 
the standard $\Z$-basis for $G$. Define $f = \sum_{i=1}^n e_i$ and set
 \[G_0 = \left\{\pm e_i,~ \pm f  \right\}.\]

Consider the monoid of zero-sum sequences $\mathcal{B}(G_0)$ over $G_0$.
The irreducible elements of $\mathcal{B}(G_0)$ are
\[\left\{e_i(-e_i),~ f (-f),~ e_1e_2
\cdots e_n(-f),~ (-e_1)(-e_2) \cdots (-e_n)f\right\}.\]
Because the supports of these sequences are pairwise incomparable, all of them are absolutely
 irreducible.

Let \[U = e_1e_2 \cdots e_n(-f) \quad\text{and}\quad V = (-e_1)(-e_2)
\cdots (-e_n)f.\] Then in $\mathcal{B}(G_0)$ we have the non-unique
factorization
\[UV = (e_1(-e_1)) \cdot (e_2(-e_2)) \cdots (e_n(-e_n)) \cdot f(-f).\]

It follows from \cite[Theorem 8]{Gilmer-Heinzer-Smith1996:dop} that 
there exists a Dedekind domain $D$ with class group $G = \Z^n$ and
$G_0$ precisely the set of classes containing prime ideals.
There exists a transfer homomorphism $\varphi \colon D\setminus\{0\} \to \mathcal{B}(G_0)$ (see Theorem~\ref{t:krull-transfer}).
In particular, the domain $D$ is not half-factorial and an element $a \in D$ 
is irreducible if and only if the corresponding zero-sum sequence $\varphi(a) \in \mathcal{B}(G_0)$ is irreducible.

Since $0 \notin G_0$, the trivial ideal class of $D$ contains no prime ideals.
Hence $D$ contains no prime element.
Finally, every irreducible element of $\mathcal B(G_0)$ is square-free,
and so Theorem~\ref{t:krull-allai} implies that every irreducible element 
of $D$ is absolutely irreducible.
\comment{%
 Finally, every irreducible element of $D$ is absolutely irreducible.
 Indeed, let $a \in D$ be irreducible. Then $aD$ factors as a product 
of non-zero prime ideals of $D$, say
 \[aD = \fp_1 \cdots \fp_m.\]
Since all the irreducible elements of $\mathcal{B}(G_0)$ are square-free, 
that is, the elements of $G_0$ do not appear with a power greater than
$1$, all the primes $\fp_j$ lie in distinct classes of $G$. For this 
reason, a non-unique factorization of $a^k$, for some positive integer
$k$, can only arise from a non-unique factorization of the element 
$[aD]^k$ in $\mathcal{B}(G_0)$. But the irreducible element 
$[aD] \in \mathcal{B}(G_0)$ is absolutely irreducible, as mentioned above.
}
\end{proof}

%% file: PAN_mpp_PA_19_08_2025.tex
\section{Rings with all irreducible elements absolutely irreducible
but not all prime}\label{section:PA}

In contrast to Proposition~\ref{Prop:withoutprime}, the following
result gives an analogous example of a Dedekind domain, but this 
time it contains a prime element.

\begin{proposition} \label{prop:exm-allai-but-not-ufd}\label{mpp}
There exists a Dedekind domain $D$ that is not half-factorial,
contains a prime element, and such that all of its
irreducible elements are absolutely irreducible.
\end{proposition}
\begin{proof}
Repeat the proof of Proposition~\ref{Prop:withoutprime} with the set
\[G_0 = \left\{0, \pm e_i,~ \sum_{i = 1}^ne_i,~ \sum_{i = 1}^n-e_i \right\}.\]
Note that $0 \in G_0$ and hence there exists a non-zero principal 
prime ideal in $D$ and therefore a prime element.
\end{proof}